\documentclass[a4paper,11pt]{amsart}
\usepackage{amssymb}
\usepackage{amsrefs}
\usepackage{mathrsfs}
\usepackage{verbatim}
\usepackage[colorlinks]{hyperref}
\usepackage{graphicx}

\newcommand{\Gpre}{G}
\newcommand{\T}{^{\scriptscriptstyle\top}}
\newcommand{\zbar}{\bar{z}}

\newcommand{\Gc}{\mathscr{G}}

\newcommand{\R}{\mathbb{R}}
\newcommand{\Z}{\mathbb{Z}}

\DeclareMathOperator{\sgn}{sgn}

\newcounter{cond}
\newcommand{\condition}[1]
	{\stepcounter{cond}\tag{C\arabic{cond}}\label{#1}}

\newtheorem*{Motzkin}
	{Motzkin's Theorem {\mdseries\cite{MR0223384}*{Theorem~8}}}
\newtheorem*{Farkas}{Farkas' Lemma}

\newtheorem{theorem}{Theorem}
\newtheorem{lemma}[theorem]{Lemma}

\theoremstyle{definition}
\newtheorem*{remark}{Remark}

\let\bmin\wedge

\let\empty\varnothing

\title[Matrices with cyclicly decreasing rows]{On the class of matrices with 
rows that weakly decrease cyclicly from the diagonal}

\author{Wouter Kager$^\dagger$}
\address{$^\dagger$Vrije Universiteit Amsterdam, Department of Mathematics,
	Amsterdam, the Netherlands}

\author{Pieter Jacob Storm$^\ddagger$}
\address{$^\ddagger$Eindhoven University of Technology,
	Department of Mathematics and Computer Science,
	Eindhoven, the Netherlands}



\date{\today}

\begin{document}

\begin{abstract}
	We consider $n\times n$ real-valued matrices~$A = (a_{ij})$ satisfying 
	$a_{ii} \geq a_{i,i+1} \geq \dots \geq a_{in} \geq a_{i1} \geq \dots \geq 
	a_{i,i-1}$ for $i=1,\dots,n$. With such a matrix~$A$ we associate a 
	directed graph~$\Gc(A)$. We prove that the solutions to the system $A\T x 
	= \lambda \, e$, with $\lambda \in\R$ and $e$ the vector of all ones, are 
	linear combinations of `fundamental' solutions to~$A\T x=e$ and vectors 
	in~$\ker A\T$, each of which is associated with a closed strongly 
	connected component (SCC) of~$\Gc(A)$. This allows us to characterize the 
	sign of~$\det A$ in terms of the number of closed SCCs and the solutions 
	to~$A\T x = e$. In addition, we provide conditions for $A$ to be a 
	$P$-matrix.
\end{abstract}

\maketitle

\section{Introduction}

We consider the class of real $n\times n$ matrices with rows that weakly 
decrease in a cyclic fashion from the diagonal, i.e., matrices $A = (a_{ij})$ 
that satisfy
\begin{equation}
	\condition{cnd: non-incr}
	a_{ii} \geq a_{i,i+1} \geq \dots \geq a_{in}
	\geq a_{i1} \geq \dots \geq a_{i,i-1}, \qquad
	i=1,\dots,n.
\end{equation}
Our interest lies in characterizing~$\sgn(\det A)$, the sign of the 
determinant of such matrices, with the convention that $\sgn 0 = 0$. In our 
exposition, we regard the dimension~$n$ as fixed and denote the vector of all 
ones by~$e$.

\subsection{Statement of the main result}
\label{sec: main result}

Consider an $n\times n$ matrix~$A$ that satisfies~\eqref{cnd: non-incr}. We 
say there is a \emph{gap} right before the matrix element~$a_{ij}$ if $a_{ij}$ 
is strictly smaller than the previous element~$a_{i,j-1}$ ($a_{in}$ in case 
$j=1$) in the same row. Let $K = (\kappa_{ij})$ be the 0\,--1 matrix obtained 
from~$A$ by replacing each matrix element right after a gap by a~1, and every 
other element by a~0, that is, $K$ is defined by
\[
	\kappa_{ij} := \begin{cases}
		1 & \text{if $j=1$ and $a_{in} > a_{i1}$,
			or $2\leq j\leq n$ and $a_{i,j-1} > a_{ij}$}; \\
		0 & \text{otherwise.}
	\end{cases}
\]

We interpret~$K$ as the adjacency matrix of a directed graph with vertex 
set~$\{1,\dots,n\}$ and a directed edge from~$i$ to~$j$ if and only if 
$\kappa_{ij}=1$, and call this graph~$\Gc(A)$. We say two distinct vertices 
$i$ and~$j$ are \emph{strongly connected} if there are paths in~$\Gc(A)$ from 
$i$ to~$j$ and from $j$ to~$i$, and declare each vertex~$i$ to be strongly 
connected to itself. This defines an equivalence relation on the vertex set, 
the equivalence classes of which we call \emph{strongly connected components} 
(SCCs). A strongly connected component~$C$ is called \emph{open} if the 
graph~$\Gc(A)$ contains a directed edge from a vertex in~$C$ to a vertex not 
in~$C$, and \emph{closed} otherwise. Necessarily, at least one of the SCCs of 
the graph~$\Gc(A)$ has to be closed. Figure~\ref{fig: directed graph} 
illustrates these definitions.

The main result of this paper is that one of the following four possibilities 
must apply to the matrix~$A$:
\begin{enumerate}
	\item[1.] There are at least two closed SCCs and $\det A=0$.
	\item[2.] There is exactly one closed SCC and either
		\begin{enumerate}
			\item the system $A\T x=e$ has no solution and $\det A=0$;
				or
			\item the system $A\T x=e$ has a solution $x\geq0$ and $\det A>0$;
				or
			\item the system $A\T x=e$ has a solution $x\leq0$ and $\det A<0$.
		\end{enumerate}
\end{enumerate}
We also show using Farkas' lemma that 2(a) and~2(c) are impossible if $Ae>0$, 
while 2(a) and~2(b) are ruled out if $Ae<0$. So if the row sums of~$A$ are all 
strictly positive or all strictly negative, then whether or not $A$ is 
singular is determined by the number of closed SCCs of~$\Gc(A)$.

\begin{figure}
	\begin{center}
		$\displaystyle
			A = \begin{bmatrix} \,
				2 & 1 & 1 & 1 & 0 \\
				0 & 1 & 1 & 0 & 0 \\
				2 & 1 & 2 & 2 & 2 \\
				2 & 1 & 0 & 2 & 2 \\
				1 & 1 & 1 & 0 & 2 \,
			\end{bmatrix}
		\quad
			K = \begin{bmatrix} \,
				0 & 1 & 0 & 0 & 1 \\
				0 & 0 & 0 & 1 & 0 \\
				0 & 1 & 0 & 0 & 0 \\
				0 & 1 & 1 & 0 & 0 \\
				1 & 0 & 0 & 1 & 0 \,
			\end{bmatrix}
		\quad
		\parbox[0]{95bp}{\includegraphics{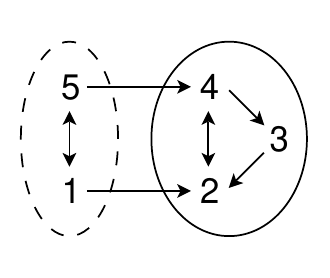}}$
	\end{center}
	\begin{center}
		$\displaystyle
			A = \begin{bmatrix} \,
				2 & 2 & 2 & 1 & 1 \\
				1 & 3 & 3 & 2 & 2 \\
				0 & 0 & 1 & 1 & 0 \\
				2 & 2 & 2 & 2 & 2 \\
				1 & 1 & 0 & 0 & 1 \,
			\end{bmatrix}
		\quad
			K = \begin{bmatrix} \,
				0 & 0 & 0 & 1 & 0 \\
				1 & 0 & 0 & 1 & 0 \\
				0 & 0 & 0 & 0 & 1 \\
				0 & 0 & 0 & 0 & 0 \\
				0 & 0 & 1 & 0 & 0 \,
			\end{bmatrix}
		\quad
		\parbox[0]{95bp}{\includegraphics{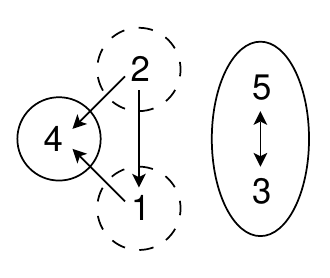}}$
	\end{center}
	\caption{Two matrices~$A$ in the class~\eqref{cnd: non-incr} with the 
	corresponding adjacency matrix~$K$ and directed graph~$\Gc(A)$. The open 
	strongly connected components are encircled with a dashed line, the closed 
	ones with a solid line. For the first matrix we have $\det A>0$, for the 
	second $\det A=0$.}
	\label{fig: directed graph}
\end{figure}

Our main result generalizes the following theorem by T.S.\ Motzkin:

\begin{Motzkin}
	If the real $n\times n$ matrix~$A$ is non-negative and 
	satisfies~\eqref{cnd: non-incr}, then $\det A\geq0$, and if the 
	inequalities in~\eqref{cnd: non-incr} are strict, then $\det A>0$.
\end{Motzkin}

Compared to Motzkin's theorem, in the case of a non-negative matrix~$A$ our 
result provides a precise condition on which of the inequalities 
in~\eqref{cnd: non-incr} need to be strict for~$\det A$ to be strictly 
positive. Moreover, our result also covers matrices in the class~\eqref{cnd: 
non-incr} with negative entries.

Although we focus on matrices with weakly decreasing rows, our analysis can 
also be used if the matrix~$A$ has rows that either weakly decrease or weakly 
increase from the diagonal in a cyclic fashion. Indeed, for such a matrix 
there exists a diagonal matrix~$D$ with diagonal entries $1$ and~$-1$ such 
that the product~$DA$ satisfies~\eqref{cnd: non-incr}. The analysis of our 
paper applied to the matrix~$DA$ then translates into results for the 
matrix~$A$. A somewhat related class of matrices with rows that decrease from 
the diagonal in both directions has been studied by Rousseau 
in~\cite{MR1674216}.

\subsection{Application and motivation}
\label{sec: application}

Our motivation for considering the class of matrices satisfying~\eqref{cnd: 
non-incr} comes from a model used to study particle flows in ring-topology 
networks. Such network structures occur for instance in traffic and 
communication networks. Here we describe only the main characteristics of the 
model; we refer to~\cite{storm22} for a more detailed account of the dynamics 
and an extended motivation for studying ring-topology networks.

The system consists of $n$~stations and $n$~cells arranged in a ring, with 
cell~1 followed by station~1, cell~2, station~2, etc., up to station~$n$. Each 
cell can contain at most one particle at a time. At each station~$i$, 
particles arrive from outside at a fixed rate~$p_i$, and are placed in a 
buffer before they can enter the ring. At integer times, each station~$i$ can 
perform one of three actions: if cell~$i$ contains a particle, the station 
either forwards this particle to the next cell or removes it from the ring; if 
cell~$i$ is empty but the buffer at station~$i$ is not, the station moves one 
particle from the buffer to the next cell; otherwise, the station does 
nothing. This is illustrated in Fig.~\ref{fig: model}.

\begin{figure}
	\includegraphics[scale=0.94]{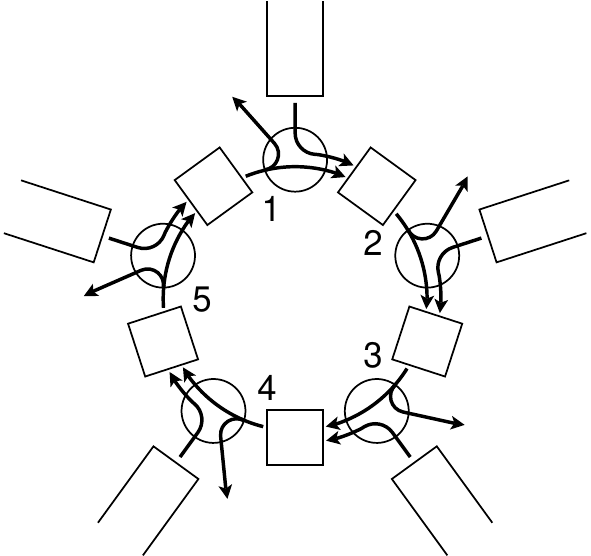}
	\caption{The $n=5$ ring-topology network. The circles are the stations, 
	the squares represent the cells, and the arrows correspond to the three 
	actions each station can perform at integer times (see Section~\ref{sec: 
	application} for details).}
	\label{fig: model}
\end{figure}

Consider a particle moving onto the ring at station~$i$. We assume the time it 
spends on the ring before being removed follows a probability distribution 
depending only on~$i$. Let $a_{ij}$ be the expected number of times the 
particle passes through station~$j$, and let $b_{ij}$ be the expected number 
of times it occupies cell~$j$. Then, by the nature of the model, the $n\times 
n$ matrix~$A = (a_{ij})$ satisfies~\eqref{cnd: non-incr} and has the relation 
$A = I+B$ with the matrix~$B = (b_{ij})$.

We are interested in determining the possible stationary particle flows 
through the network. To this end, suppose that for each cell~$i$ there is a 
stationary rate~$\pi_i$ at which the cell is vacant. Then the stationary rate 
at which station~$i$ moves particles from the buffer onto the ring is the 
smaller of $\pi_i$ and~$p_i$, since the station can only move a particle onto 
the ring when cell~$i$ is empty. Using that $b_{ij}$ is the expected number of 
times such a particle occupies cell~$j$, it follows that the vector~$\pi = 
(\pi_i)$ has to satisfy
\begin{equation}
	\label{eqn: throughput}
	\pi = e - B\T (\pi\bmin p),
\end{equation}
where $p = (p_i)$ is the vector of arrival rates, and $\pi\bmin p$ is the 
component-wise minimum of $\pi$ and~$p$. This is called the throughput 
equation. Its solutions for~$\pi$ characterize all the candidates for a 
stationary particle flow.

In a separate paper~\cite{kager22} we show that every solution to~\eqref{eqn: 
throughput} can be expressed in terms of the solutions to the system $A\T x=e$ 
and analogous systems for principal submatrices of~$A$. A solution always 
exists, but whether it is unique depends on the invertibility of these 
matrices. In particular, we prove in~\cite{kager22} that the throughput 
equation has a unique solution for every vector~$p$ if and only if $A$ is a 
$P$-matrix. The current paper supplements~\cite{kager22} by studying the 
solutions to the system $A\T x= \lambda\,e$, for any $\lambda \in\R$, and 
characterizing the sign of~$\det A$ (as explained above) for matrices in the 
class~\eqref{cnd: non-incr}, in Section~\ref{sec: results}. In 
Section~\ref{sec: P-matrix} we provide conditions under which (non-negative) 
matrices in this class are $P$-matrices, considering their importance for our 
application. Section~\ref{sec: conclusions} summarizes the results and 
conclusions of the paper.

\subsection{Notation and terminology}
\label{sec: notation}

Vectors and matrices are real-valued, all vectors are column vectors, and 
vector and matrix inequalities are to be read component-wise. The typical 
matrix and vector dimensions are $n\times n$ and~$n$. Given any integer~$i$, 
we write~$[i]$ for the unique index in~$\{1,\dots,n\}$ for which $i = kn + 
[i]$ is true for some integer~$k$. In particular, $[kn+i] = i$ whenever $i\in 
\{1, \dots, n\}$ and $k\in\Z$. Also, for any two indices~$i,j$ in the 
set~$\{1, \dots, n\}$, $d(i,j)$ denotes the unique number~$k$ in~$\{0, \dots, 
n-1\}$ for which $[i+k] = j$. That is, $d(i,j)$ is the number of indices you 
have to move forward in the cyclic order on~$\{1, \dots, n\}$ to reach~$j$ 
starting from~$i$.

The $n \times n$ matrix of all ones is denoted by~$E$. If $A$ is an $n\times 
n$ matrix and $I,J$ are subsets of $\{1,\dots,n\}$, then $[A]_{IJ}$ is the 
submatrix of~$A$ consisting of the entries~$(\, a_{ij}\colon i\in I, j\in J 
\,)$. Similarly, if $x$ is an $n$-dimensional vector, then $[x]_I$ is the 
subvector consisting of the components~$(\, x_i\colon i\in I\,)$.

A matrix~$A$ is called: a \emph{$P$-matrix} if all its principal minors are 
strictly positive; a \emph{$Z$-matrix} if $a_{ij}\leq0$ for $i\neq j$; an 
\emph{$M$-matrix} if $A$ can be expressed in the form $sI-B$ with $B\geq0$ and 
$s\geq \rho(B)$; \emph{semi-positive} if $Ax>0$ for some vector~$x>0$. A 
result we use is that a semi-positive $Z$-matrix is a non-singular 
$M$-matrix~\cite {MR444681}* {Theorem~1, Condition~K33}.

Finally, throughout the paper, $K = (\kappa_{ij})$ will be the adjacency 
matrix of the directed graph~$\Gc(A)$ associated with the matrix~$A$ under 
consideration, as defined in Section~\ref{sec: main result}.

\section{Main results and proofs}
\label{sec: results}

The basis of our approach is an investigation into the solutions of the system 
$A\T x = \lambda\,e$, with $\lambda \in\R$. We first show that these solutions 
are zero outside the union of the closed SCCs of the graph~$\Gc(A)$. We then 
investigate the `fundamental' solutions to~$A\T x=e$ that are characterized by 
the fact that they are supported on a single closed SCC. We prove that the 
solutions to $A\T x = \lambda\,e$ are linear combinations of these fundamental 
solutions and vectors in~$\ker A\T$ that also are associated with a particular 
closed SCC. We then establish that every closed SCC admits at most one 
fundamental solution, which is either non-positive or non-negative. This 
finally leads to the characterization of the sign of~$\det A$ stated in 
Section~\ref{sec: main result}.

\subsection{The support of solutions}

We begin our exposition by showing that the solutions to $A\T x = \lambda\,e$ 
are supported on the closed SCCs of~$\Gc(A)$. To do so, we make use of the 
notations introduced in Section~\ref{sec: notation}.

\begin{lemma}
	\label{lem: support}
	Let $\lambda \in \R$. Suppose the real $n\times n$ matrix~$A$ 
	satisfies~\eqref{cnd: non-incr} and the vector~$x$ solves $A\T x = \lambda 
	\, e$. Let $D$ be the union of the open SCCs of the graph~$\Gc(A)$, and 
	suppose $D$ is not empty. Then $[x]_D = 0$.
\end{lemma}

\begin{proof}
	Let~$C$ be the union of the closed SCCs of the graph~$\Gc(A)$, so that $C 
	= \{1,\dots,n\} \setminus D$, and define the $n\times n$ matrix~$M = 
	(m_{ij})$ by
	\[
		m_{ij} := a_{ij} - a_{i[j-1]},
		\qquad i,j = 1,\dots,n.
	\]
	Observe that if $i\in C$ and~$j\in D$, then we must have $m_{ij} = 0$ 
	because there cannot be an edge in the graph~$\Gc(A)$ from $i$ to~$j$. It 
	follows that, for $j\in D$,
	\[\begin{split}
		\sum_{i\in D} m_{ij}\,x_i
		= \sum_{i=1}^n m_{ij}\,x_i
		&= \sum_{i=1}^n a_{ij}\,x_i - \sum_{i=1}^n a_{i[j-1]}\,x_i \\
		&= (A\T x)_j - (A\T x)_{[j-1]} = 0,
	\end{split}\]
	hence $[M]\T_{DD} [x]_D = 0$. We claim that $[M]_{DD}$ is a non-singular 
	$M$-matrix, and since this implies $[x]_D = 0$ it only remains to prove 
	this claim.

	Note that $Me=0$ by the definition of~$M$. It is now easy to see 
	that~$[M]_{DD}$ is a $Z$-matrix with strictly positive diagonal elements: 
	by \eqref{cnd: non-incr} we have $m_{ij} \leq 0$ whenever $i\neq j$ with 
	strict inequality if $\kappa_{ij}=1$, and for $i$ in an open SCC 
	of~$\Gc(A)$ we then obtain $m_{ii} > 0$ from $(Me)_i=0$ and the fact that 
	there must be \emph{some} index~$j$ other than~$i$ such that $\kappa_{ij} 
	= 1$. As for the row sums of the matrix~$[M]_{DD}$, $Me = 0$ together with 
	$m_{ij} \leq 0$ for $i\neq j$ yields
	\[
		[M]_{DD}[e]_D
		= [Me]_D - [M]_{DC}[e]_C
		= - [M]_{DC}[e]_C
		\geq 0.
	\]
	So each row sum of~$[M]_{DD}$ is non-negative, and the sum of the $i$-th 
	row is strictly positive if $\kappa_{ij}=1$, and hence $m_{ij}<0$, for 
	some~$j$ in~$C$.

	Let $d$ be the number of elements in~$D$. Note that every non-empty 
	subset~$I$ of~$D$ (including $I=D$) contains an index~$i$ for which there 
	is some index~$j$ outside~$I$ with $\kappa_{ij}=1$, for if not, then $I$ 
	would contain a closed SCC of~$\Gc(A)$. This allows us to recursively 
	choose the indices $i_1, \dots, i_d$ so that they are distinct and satisfy 
	the following properties: (i) $\{i_1, \dots, i_d\} = D$, (ii) for $i_1$ 
	there is an index~$j_1$ in~$C$ such that $\kappa_{i_1 j_1}=1$, and~(iii) 
	for $k = 2,\dots,d$ there is an index~$j_k$ in~$C \cup \{i_1, \dots, 
	i_{k-1}\}$ such that $\kappa_{i_k j_k}=1$. From our observations in the 
	previous paragraph we then see that for each~$i_k$ we have
	\begin{equation}
		\label{eqn: M-matrix1}
		\sum_{\ell=1}^d m_{i_ki_\ell} > 0
		\quad\text{or}\quad
		m_{i_ki_\ell} < 0
		\quad\text{for some~$\ell$ in~$\{1,\dots,k\} \setminus \{k\}$}.
	\end{equation}

	We now define the vector~$z = (z_i\colon i\in D)$ recursively by setting
	\[
		z_{i_k}
		:= 1 - \frac1{2m_{i_ki_k}} \Biggl(\,
				\sum_{\ell=1}^{k-1} m_{i_ki_\ell}\,z_{i_\ell}
				+ \sum_{\ell=k}^d m_{i_ki_\ell}
			\Biggr), \quad k = 1,\dots,d,
	\]
	where for $k=1$ the first sum is zero. We claim that by induction in~$k$, 
	it follows from this definition that for $k = 1,\dots,d$,
	\[
		0
		< \sum_{\ell=1}^{k-1} m_{i_ki_\ell}\,z_{i_\ell}
			+ \sum_{\ell=k}^d m_{i_ki_\ell}
		\leq m_{i_ki_k}
		< 2m_{i_ki_k}
		\quad\text{and}\quad
		0<z_{i_k}<1.
	\]
	Indeed, this is not difficult to verify using~\eqref{eqn: M-matrix1} and 
	the fact that $[M]_{DD}$ is a $Z$-matrix with strictly positive diagonal 
	elements and non-negative row sums. But then it also follows that for each 
	index~$i_k$,
	\[\begin{split}
		\sum_{\ell=1}^d m_{i_ki_\ell} \, z_{i_\ell}
		&\geq \sum_{\ell=1}^{k-1} m_{i_ki_\ell} \, z_{i_\ell}
			+ \sum_{\ell=k}^d m_{i_ki_\ell}
			- m_{i_ki_k} \, ( 1-z_{i_k} ) \\
		&= 2m_{i_ki_k} ( 1-z_{i_k} ) - m_{i_ki_k} ( 1-z_{i_k} ) > 0.
	\end{split}\]
	So $z>0$ and $[M]_{DD}\,z>0$, hence $[M]_{DD}$ is semi-positive. Since 
	every semi-positive $Z$-matrix is a non-singular $M$-matrix, we are 
	done.
\end{proof}

\subsection{The fundamental solutions}

We now start our investigation into the fundamental solutions to $A\T x=e$ and 
their relation with the solutions to the general system $A\T x = \lambda\,e$ 
with $\lambda\in\R$. We begin with a lemma concerning the non-negative 
solutions to the general system, in the proof of which we use the definition 
of the numbers~$d(i,j)$ from Section~\ref{sec: notation}.

\begin{lemma}
	\label{lem: class support}
	Let $\lambda \in \R$. Suppose the real $n\times n$ matrix~$A$ 
	satisfies~\eqref{cnd: non-incr} and the vector~$x$ is a solution to the 
	system $A\T x = \lambda \, e$ with $x\geq0$ and $x\neq0$. Then there is at 
	least one closed SCC of~$\Gc(A)$, say~$C$, such that $[x]_C > 0$.
\end{lemma}

\begin{proof}
	Note that there must be an index~$r$ such that $x_r>0$, and by 
	Lemma~\ref{lem: support}, $r$ lies in a closed SCC. We will prove below 
	that if $x_r>0$ and $\kappa_{rj}=1$, then $x_j>0$. This suffices to prove 
	the lemma because it implies $x_j>0$ for every index~$j$ to which there is 
	a path from~$r$ in the graph~$\Gc(A)$.

	So assume $x_r>0$ and $\kappa_{rj}=1$. Suppose, towards a contradiction, 
	that $x_j=0$. Now let $I$ be the set of indices~$i$ for which $x_i>0$, and 
	let $k$ be the element of~$I$ for which $d(k,j)$ is minimal. Note that 
	by~\eqref{cnd: non-incr} we then have $a_{ij} \leq a_{ik}$ for every 
	element~$i$ of~$I$ because $d(i,k) < d(i,j)$. Moreover, since there is a 
	gap just before element~$a_{rj}$ in row~$r$, we have $a_{rj} < a_{rk}$. 
	Therefore,
	\[
		(A\T x)_j
		= \sum_{i=1}^n a_{ij}\,x_i
		= \sum_{i\in I} a_{ij}\,x_i
		< \sum_{i\in I} a_{ik}\,x_i
		= (A\T x)_k.
	\]
	But this contradicts $A\T x = \lambda\,e$, so it must be the case that 
	$x_j>0$.
\end{proof}

We now focus on solutions to the system $A\T x = \lambda\,e$ that are 
supported on a specific closed SCC of the graph~$\Gc(A)$. Our next two lemmas 
complement each other. The first says that if the vector~$x$ solves $A\T x = 
\lambda\,e$ and $C$ is a closed SCC of~$\Gc(A)$, then the subvector $[x]_C$ is 
mapped to a constant vector by the submatrix~$[A]\T_{CC}$. Conversely, the 
second lemma says that if $x$ is supported on~$C$ and the subvector~$[x]_C$ is 
mapped to the constant vector~$[\lambda\,e]_C$ by~$[A]\T_{CC}$, then $x$ 
solves the system~$A\T x = \lambda\,e$.

\begin{lemma}
	\label{lem: sub solution}
	Let $\lambda\in\R$. Suppose the real $n\times n$ matrix~$A$ 
	satisfies~\eqref{cnd: non-incr}, $C$ is a closed SCC of~$\Gc(A)$, and $x$ 
	is a vector that solves the system $A\T x = \lambda\,e$. Then $[A]\T_{CC} 
	\, [x]_C = [\,\mu\,e]_C$ for some~$\mu\in\R$.
\end{lemma}

\begin{proof}
	If $C$ has only one element there is nothing to prove, so assume $C$ has 
	at least two elements. Let $\{i_1,\dots,i_k\}$ be the set of all indices 
	that are element of a closed SCC of~$\Gc(A)$, where $i_1 <i_2 <\dots 
	<i_k$. Now choose two indices $i_\ell$ and~$i_m$ from this set that both 
	lie in~$C$, such that $i_\ell<i_m$ and no index between $i_\ell$ and~$i_m$ 
	lies in~$C$. Then if $j$ is an index in~$C$ we have $a_{ji_\ell} = 
	a_{ji_{m-1}}$ because the gaps on row~$j$ of the matrix~$A$ occur only 
	right before elements with a column index in~$C$. Similarly, if $j$ is an 
	index in a closed SCC other than~$C$, then $a_{ji_{m-1}} = a_{ji_m}$ 
	because the gaps on row~$j$ occur only right before elements~$a_{ji_k}$ 
	with $i_k$ in the same SCC as~$j$. Finally, by Lemma~\ref{lem: support}, 
	if $j$ is an index in an open SCC, then $x_j=0$. It follows that
	\[\begin{split}
		\sum_{j\in C} \bigl( a_{ji_\ell} - a_{ji_m} \bigr) \, x_j
		&= \sum_{j\in C} \bigl( a_{ji_{m-1}} - a_{ji_m} \bigr) \, x_j \\
		&= \sum_{j=1}^n \bigl( a_{ji_{m-1}} - a_{ji_m} \bigr) \, x_j
		 = (A\T x)_{i_{m-1}} - (A\T x)_{i_m}
		 = 0,
	\end{split}\]
	so the components of the vector $[A]\T_{CC} \, [x]_C$ all have the same 
	value.
\end{proof}

\begin{lemma}
	\label{lem: class solution}
	Let $\lambda\in\R$. Suppose the real $n\times n$ matrix~$A$ 
	satisfies~\eqref{cnd: non-incr}, $C$ is a closed SCC of~$\Gc(A)$, and $x$ 
	is a vector that solves the system
	\[
		[A]\T_{CC} \, [x]_C = [\lambda\,e]_C,\qquad [x]_{C^c} = 0,
	\]
	where $C^c := \{1,\dots,n\} \setminus C$. Then $A\T x = \lambda\,e$.
\end{lemma}

\begin{proof}
	To see that $A\T x = \lambda\,e$, it only remains to show that $(A\T x)_i 
	= \lambda$ for $i\notin C$. So suppose $i\notin C$. Let $k$ be the element 
	of~$C$ for which $d(k,i)$ is minimal. Then we must have $a_{ji} = a_{jk}$ 
	for every~$j$ in~$C$, because row~$j$ of the matrix~$A$ cannot have any 
	gaps between columns $k$ and~$i$. Therefore,
	\[
		(A\T x)_i
		= \sum_{j=1}^n a_{ji} \, x_j
		= \sum_{j\in C} a_{ji} \, x_j
		= \sum_{j\in C} a_{jk} \, x_j
		= (A\T x)_k
		= \lambda. \qedhere
	\]
\end{proof}

We introduce some new terminology to proceed. Suppose the matrix~$A$ 
satisfies~\eqref{cnd: non-incr} and $C$ is a closed SCC of the graph~$\Gc(A)$. 
Consider the system
\begin{equation}
	\label{eqn: class system}
	[A]\T_{CC} \, [x]_C = [e]_C,
	\qquad [x]_{C^c} = 0,
\end{equation}
where $C^c := \{1,\dots,n\} \setminus C$. We call~$C$ a \emph{fundamental} 
closed SCC if this system has a solution~$x$, in which case $x$ solves $A\T 
x=e$ by Lemma~\ref{lem: class solution} and is called a \emph{fundamental 
solution for~$C$} to~$A\T x=e$. If, on the other hand, the system~\eqref{eqn: 
class system} has no solutions, then we say $C$ is \emph{null}. In that case, 
$\ker {[A]\T_{CC}}$ is non-trivial and we call any vector~$y$ such that $[y]_C 
\in \ker {[A]\T_{CC}}$ and $[y]_{C^c} = 0$ a \emph{null vector for~$C$}. Note 
that by Lemma~\ref{lem: class solution}, such a null vector lies in~$\ker 
A\T$.

We will show later that every fundamental solution is unique and either 
non-positive or non-negative. But for the moment, we only assume that for each 
fundamental closed SCC there is a fundamental solution that is non-positive or 
non-negative. Our next theorem states that then the solutions to the 
system~$A\T x = \lambda\,e$ are precisely the sums of a linear combination of 
these fundamental solutions with coefficients summing to~$\lambda$, and null 
vectors for the closed SCCs that are null. In case the closed SCCs are all 
fundamental or all null, the theorem is to be interpreted as stating that the 
solutions to the system~$A\T x = \lambda\,e$ are, respectively, linear 
combinations of only fundamental solutions or sums of only null vectors. In 
particular, the system has no solutions for $\lambda\neq0$ if every closed SCC 
is null.

\begin{theorem}
	\label{thm: solutions}
	Suppose the real $n\times n$ matrix~$A$ satisfies~\eqref{cnd: non-incr}, 
	$\Gc(A)$ has exactly~$k$ fundamental closed SCCs $C_1,\dots,C_k$ and 
	exactly~$\ell$ closed SCCs $C_{k+1},\dots,C_{k+\ell}$ that are null, and 
	for $i=1,\dots,k$ there is a non-positive or non-negative fundamental 
	solution~$x_i$ for~$C_i$. Let $\lambda\in \R$. Then the solutions to the 
	system~$A\T x = \lambda \,e$ are precisely the vectors $x = \sum_{i=1}^{k} 
	\alpha_i \,x_i + \sum_{i=1}^{\ell} y_{k+i}$ with $\sum_{i=1}^{k} \alpha_i 
	= \lambda$ and $y_{k+i}$ a null vector for~$C_{k+i}$ for $i = 
	1,\dots,\ell$.
\end{theorem}

\begin{proof}
	That every vector~$x$ of the proposed form solves the system $A\T x = 
	\lambda\,e$ is immediately clear from Lemma~\ref{lem: class solution}. To 
	prove the converse statement, let $x$ be a solution to $A\T x = 
	\lambda\,e$. For $i = 1,\dots,k+\ell$, let $y_i$ be the vector defined by 
	$[y_i]_{C_i} := [x]_{C_i}$ and $[y_i]_{C_i^c} := 0$. Then by 
	Lemma~\ref{lem: support} and the fact that the closed SCCs are disjoint, 
	we have $x = y_1 + \dots + y_{k+\ell}$. Now note that for $i = 
	1,\dots,\ell$, Lemma~\ref{lem: sub solution} implies that $y_{k+i}$ is a 
	null vector because $C_{k+i}$ is null, and then $A\T y_{k+i} = 0$ by 
	Lemma~\ref{lem: class solution}. Finally, for $i = 1,\dots,k$, using that 
	either $[x_i]_{C_i}<0$ or $[x_i]_{C_i}>0$ by Lemma~\ref{lem: class 
	support}, we set
	\[
		\alpha_i := \begin{cases}
			\phantom\sup\llap{$\inf$}
				\{\alpha\in\R \colon y_i - \alpha \, x_i \geq 0\}
			& \text{if $x_i \leq 0$}, \\
			\sup\{\alpha\in\R \colon y_i - \alpha \, x_i \geq 0\}
			& \text{if $x_i \geq 0$},
		\end{cases}
	\]
	and define $z := \sum_{i=1}^{k} (y_i - \alpha_i \, x_i)$. Note that by 
	these definitions, $z\geq0$ and for every closed SCC there is at least one 
	index~$j$ in that SCC such that $z_j=0$. Moreover, we also have
	\[
		A\T z
		= \sum_{i=1}^k A\T (y_i - \alpha_i \, x_i)
		= A\T x - \sum_{i=1}^k \alpha_i \, A\T x_i
		= \Bigl( \lambda - \sum_{i=1}^k \alpha_i \Bigr) \, e.
	\]
	By Lemma~\ref{lem: class support} it then has to be the case that $z=0$ 
	and $\sum_{i=1}^{k} \alpha_i = \lambda$.
\end{proof}

\subsection{Existence and uniqueness}

In the previous section we showed that the solutions to the system $A\T x = 
\lambda\,e$ can be expressed in terms of fundamental solutions to~$A\T x=e$, 
under the assumption that these fundamental solutions are either non-positive 
or non-negative. In this section we prove that every fundamental solution is 
unique and non-positive or non-negative. We also provide conditions that imply 
the existence of a fundamental solution for a closed SCC. We start by noting 
that the following special case of Farkas' Lemma (see, e.g., \cite{Ko1998} for 
a proof) provides a necessary and sufficient condition for the existence of a 
non-negative solution to $A\T x=e$:

\begin{Farkas}
	For a real $n\times n$ matrix~$A$ exactly one the following two assertions 
	is true:
	\begin{enumerate}
		\item[(a)]
			There exists a vector~$x$ such that $A\T x=e$ and $x \geq0$.
		\item[(b)]
			There exists a vector~$z$ such that $Az\geq0$ and $e\T z<0$.
	\end{enumerate}
\end{Farkas}

It is, however, not obvious how to assess which of the two alternatives in 
Farkas' Lemma holds for a generic matrix~$A$ in the class~\eqref{cnd: 
non-incr}. But (based on Farkas' Lemma) our next lemma shows that the system 
$A\T x=e$ has a non-negative solution if $Ae>0$ and a non-positive solution if 
$Ae<0$.

\begin{lemma}
	\label{lem: Motzkin}
	Suppose the real $n\times n$ matrix~$A$ satisfies~\eqref{cnd: non-incr}. 
	If $Ae>0$, then $Az=0$ implies $e\T z=0$ and the system $A\T x=e$ has a 
	solution~$x\geq0$. Likewise, if $Ae<0$, then $Az=0$ implies $e\T z=0$ and 
	the system $A\T x=e$ has a solution~$x\leq0$.
\end{lemma}

\begin{proof}
	This proof uses ideas from Motzkin's original paper~\cite{MR0223384}* 
	{proof of Theorem~8}. We first consider the case that $Ae>0$. Let $z$ be a 
	vector such that $e\T z = z_1 + \dots + z_n < 0$. We will show this 
	implies that at least one component of~$Az$ is strictly negative. To this 
	end, define $\zbar := \tfrac1n \sum_{i=1}^n z_i$ and let~$y$ be the vector 
	with components $y_i := z_i - \zbar$. Note that then $\zbar<0$ and 
	$y_1+\dots+y_n = 0$. Choose the row index~$r$ so that it maximizes the 
	sums $\sum_{i=1}^r y_i$. Since $\sum_{i=1}^n y_i = 0$ we then have
	\begin{equation}
		\label{eqn: Motzkin rotation}
		\sum_{i=1}^m y_{[r+i]}
		= \sum_{i=1}^{r+m} y_{[i]} - \sum_{i=1}^r y_i
		\leq 0
		\qquad \text{for $m = 1,\dots,n$}.
	\end{equation}
	Next, note that component~$[r+1]$ of the vector~$Az$ can be expressed as
	\begin{multline}
		\label{eqn: Motzkin}
		\sum_{i=1}^n a_{[r+1][r+i]} \, \bigl( y_{[r+i]} + \zbar \bigr)
		= \sum_{m=1}^{n-1} \bigl( a_{[r+1][r+m]} - a_{[r+1][r+m+1]} \bigr) 
		\sum_{i=1}^m y_{[r+i]} \\
		+ a_{[r+1][r+n]} \sum_{i=1}^n y_{[r+i]}
		+ \sum_{i=1}^n a_{[r+1][r+i]} \, \zbar.
	\end{multline}
	By \eqref{cnd: non-incr} and~\eqref{eqn: Motzkin rotation}, the first term 
	on the right is non-positive. The second term is zero because 
	$\sum_{i=1}^n y_{[r+i]} = 0$. Finally, since $Ae>0$ and~$\zbar<0$, the 
	third term is strictly negative. We conclude that $\smash{ (Az)_{[r+1]} } 
	<0$. This rules out option~(b) in Farkas' Lemma, leaving option~(a). 
	Moreover, $Az=0$ implies both $Az\geq0$ and $A(-z)\geq0$, and hence $e\T 
	z=0$ by the argument above.

	The proof in the case $Ae<0$ is similar. Again, let $z$ be a vector such 
	that $e\T z<0$, and define $\zbar$ and~$y$ as before. This time, let~$r$ 
	be the row index that \emph{minimizes} the sums $\sum_{i=1}^r y_i$, so 
	that
	\[
		\sum_{i=1}^m y_{[r+i]}
		= \sum_{i=1}^{r+m} y_{[i]} - \sum_{i=1}^r y_i
		\geq 0
		\qquad \text{for $m = 1,\dots,n$}.
	\]
	From equation~\eqref{eqn: Motzkin} we conclude that $(Az)_{[r+1]} >0$. 
	Farkas' Lemma applied to the matrix~$-A$ then says that there exists a 
	non-negative vector~$x$ such that $A\T(-x)=e$. Also, $Az=0$ again implies 
	$e\T z=0$, as before.
\end{proof}

Lemma~\ref{lem: Motzkin} does not specifically target \emph{fundamental} 
solutions to~$A\T x = e$, which are supported on a particular closed SCC of 
the graph~$\Gc(A)$. But if the matrix~$A$ satisfies~\eqref{cnd: non-incr} and 
$C$ is a closed SCC of~$\Gc(A)$, then the submatrix~$[A]_{CC}$ also 
satisfies~\eqref{cnd: non-incr}. Therefore, Lemma~\ref{lem: Motzkin} combined 
with Lemma~\ref{lem: class solution} applied to the submatrix~$[A]_{CC}$ 
provides simple conditions for the existence of a non-positive or non-negative 
fundamental solution for~$C$.

Now observe that because $C$ is a closed SCC, any gaps on a row of the 
matrix~$A$ that has a row index in~$C$ occur right before elements with a 
column index in~$C$. It follows that the adjacency matrix associated with the 
submatrix~$[A]_{CC}$ is simply~$[K]_{CC}$. As a consequence, all vertices of 
the graph~$\Gc( [A]_{CC} )$ are mutually strongly connected. Our next lemma 
applied to the submatrix~$[A]_{CC}$ therefore shows that if $C$ is 
fundamental, then the fundamental solution for~$C$ is unique and 
non-positive or non-negative.

\begin{lemma}
	\label{lem: single class solution}
	Suppose the real $n\times n$ matrix~$A$ satisfies~\eqref{cnd: non-incr} 
	and $\{1,\dots,n\}$ is the only SCC of~$\Gc(A)$. Then the system $A\T x = 
	e$ has at most one solution, and if $x$ is a solution, then either $x<0$ 
	or $x>0$.
\end{lemma}

\begin{proof}
	Suppose $x_1$ and~$x_2$ are solutions to the system~$A\T x = e$. Recall 
	that~$E$ is the $n\times n$ matrix of all ones, and consider the matrices 
	$A + \lambda E$ for $\lambda>0$. These matrices clearly 
	satisfy~\eqref{cnd: non-incr}, and because they have gaps in the same 
	positions on all rows, $\{1,\dots,n\}$ is their only SCC. Moreover, 
	$A+\lambda E>0$ for $\lambda$ large enough. So we can choose~$\lambda$ 
	such that $A + \lambda E > 0$ and, in addition, $\lambda \,e\T x_1 \neq 
	-1$ and $\lambda \,e\T x_2\neq -1$. Then by Lemma~\ref{lem: Motzkin} and 
	Theorem~\ref{thm: solutions}, there is a unique non-negative vector~$y$ 
	that satisfies $(A + \lambda E)\T y = e$. By Lemma~\ref{lem: class 
	support}, $y>0$. But for $i=1,2$ we also have
	\[
		(1+\lambda\, e\T x_i)^{-1} (A+\lambda E)\T x_i
		= (1+\lambda\, e\T x_i)^{-1} (1+\lambda \,e\T x_i) \,e = e.
	\]
	It follows that $x_1 = \mu_1 \,y$ and $x_2 = \mu_2 \,y$, where $\mu_i = 
	(1+\lambda \,e\T x_i)$. But then $A\T x_1 = A\T x_2 = e$ implies $\mu_1 = 
	\mu_2 \neq 0$. So $x_1=x_2$, and since $y>0$, it is the case that either 
	$x_1 < 0$ or $x_1 > 0$.
\end{proof}

\subsection{Main result}

We are now in a position to prove our main result. We separate this into two 
theorems, the first one covering possibilities 1 and~2(a) from 
Section~\ref{sec: main result}, and the second covering possibilities 2(b) 
and~2(c).

\begin{theorem}
	\label{thm: singular}
	The real $n\times n$ matrix~$A$ is singular if it satisfies~\eqref{cnd: 
	non-incr} and~$\Gc(A)$ has at least two closed SCCs or exactly one closed 
	SCC which is null.
\end{theorem}

\begin{proof}
	Note that under the assumptions of the theorem, the graph~$\Gc(A)$ has at 
	least one closed SCC which is null, or at least two closed SCCs that are 
	fundamental. Now if $C$ is any closed SCC of~$\Gc(A)$, then by 
	Lemma~\ref{lem: single class solution} applied to the submatrix~$[A]_{CC}$ 
	there either is a unique fundamental solution for~$C$ which is 
	non-positive or non-negative, or $C$ is null. If $C$ is null, then $\ker 
	{[A]_{CC}}$ is non-trivial and there are infinitely many null vectors 
	for~$C$. It therefore follows from Theorem~\ref{thm: solutions} that the 
	system~$A\T x=0$ has infinitely many solutions, hence $\det A\T=\det A=0$.
\end{proof}

\begin{theorem}
	\label{thm: invertible}
	Suppose the real $n\times n$ matrix~$A$ satisfies~\eqref{cnd: non-incr}, 
	the graph~$\Gc(A)$ has exactly one closed SCC, and the 
	system~$A\T x = e$ has a solution~$x$. Then either $x\geq0$ and $\det 
	A>0$, or $x\leq0$ and $\det A<0$.
\end{theorem}

\begin{proof}
	Let~$C$ be the closed SCC of~$\Gc(A)$, and let~$D$ be the union of the 
	open SCCs of~$\Gc(A)$. By Lemma~\ref{lem: support}, $[x]_D = 0$ if $D$ is 
	non-empty, so $x$ must be a fundamental solution for~$C$. Lemma~\ref{lem: 
	single class solution} applied to the submatrix~$[A]_{CC}$ yields 
	$[x]_C<0$ or $[x]_C>0$. Theorem~\ref{thm: solutions} then tells us that 
	$A\T z=0$ implies $z=0$, hence $\det A\neq0$. It remains to determine the 
	sign of~$\det A$.

	We start with the case $x\geq0$. Let $C$ (the closed SCC) consist of the 
	indices $i_1,\dots,i_k$, where $i_1 <i_2 <\dots <i_k$. Define the $n\times 
	n$ matrix~$B = (b_{ij})$ by
	\[
		b_{ij} := \begin{cases}
			1 & \text{if $i\in D$ and $j=i$}; \\
			x_i^{-1} & \text{if $i=i_\ell\in C$ and $j \in \bigl\{ i_\ell, 
			[i_\ell+1], \dots, [i_{\ell+1}-1] \bigr\}$}; \\
			0 & \text{otherwise};
		\end{cases}
	\]
	where $i_{k+1}$ is to be interpreted as~$i_1$. Note that the matrix~$B$ 
	satisfies~\eqref{cnd: non-incr} by definition. Moreover, in the 
	corresponding graph~$\Gc(B)$ every vertex~$i$ has exactly one outgoing 
	edge: for $i\in D$ this edge links~$i$ to the next vertex~$[i+1]$, and for 
	$i=i_\ell \in C$ it links~$i$ to~$i_{\ell+1}$, which is the next vertex 
	in~$C$. It follows that~$C$ is the only closed SCC for the 
	matrix~$B$.

	Next, observe that the non-zero entries on the rows of~$B$ with an index 
	in~$C$ span exactly all the columns of the matrix. That is, for every 
	column index~$j$ there is exactly one row index~$i$ in~$C$ such that 
	$b_{ij}$ is non-zero, and in fact $b_{ij} = x_i^{\smash{-1}}$. From this 
	it follows immediately that
	\[
		(B\T x)_j = \sum_{i\in C} b_{ij}\,x_i = 1,
	\]
	hence $B\T x=e$. Furthermore, if $j$ is an index in~$C$, then $b_{jj}$ is 
	the only non-zero element in column~$j$ of the matrix~$B$, and if $i$ is 
	an index in~$D$, then the only non-zero element on row~$i$ is a~1 on the 
	diagonal. This means that we can bring the matrix~$B$ in lower triangular 
	form using elementary column operations that leave the diagonal elements 
	alone, and it follows that
	\[
		\det B
		= \prod_{i=1}^n b_{ii}
		= \prod_{\ell=1}^k \frac1{x_{i_\ell}}
		> 0.
	\]

	Now for $t$ in~$[0,1]$, let $A_t$ be the matrix $(1-t)A + tB$. Observe 
	that if~$t$ is strictly between 0 and~1, then there is a gap right before 
	the element~$(A_t)_{ij}$ in row~$i$ of the matrix~$A_t$ if and only if 
	there is a gap right before~$a_{ij}$ in~$A$ or right before~$b_{ij}$ 
	in~$B$. It follows that for each of the matrices~$A_t$, with $t\in[0,1]$, 
	$C$ is the only closed SCC of the graph~$\Gc(A_t)$. Clearly, it is also 
	the case that $x$ satisfies $A_t\T x = e$, so by Theorem~\ref{thm: 
	solutions} it follows that $A_t\T z=0$ implies $z=0$. Hence, $\det A_t 
	\neq0$ for each~$t$ in~$[0,1]$, and since $\det A_t$ is a continuous 
	function of~$t$, the determinants of the matrices~$A_t$ must all have the 
	same sign. In particular, $\det A>0$ because $\det B>0$.

	We now move on to the case~$x\leq0$. We can use essentially the same 
	argument as before, but because $x_i$ is now negative for $i\in C$ we need 
	to define the matrix~$B$ differently. To be specific, we now set
	\[
		b_{ij} := \begin{cases}
			-1 & \text{if $i\in D$ and $j=[i-1]$}; \\
			x_i^{-1} & \text{if $i=i_\ell\in C$ and $j \in \bigl\{ i_{\ell-1}, 
			[i_{\ell-1}+1], \dots, [i_\ell-1] \bigr\}$}; \\
			0 & \text{otherwise};
		\end{cases}
	\]
	where $i_0$ is to be interpreted as~$i_k$. In the graph~$\Gc(B)$, this 
	links every index~$i$ in~$D$ exclusively to the previous index~$[i-1]$, 
	and every index in~$C$ exclusively to the previous index in~$C$. So again, 
	$C$ is the only closed SCC for the matrix~$B$, and it is 
	also not difficult to see that $B\T x=e$ again. It then follows by the 
	same argument as before that $\det A$ and~$\det B$ have the same sign. So 
	all that remains is to compute $\det B$.

	To compute $\det B$, we first permute the columns so that column~$n$ 
	becomes the first column and each of the columns $1,\dots,n-1$ is moved 
	one position to the right. Note that this permutation can be carried out 
	using $n-1$ column exchanges, and results in the matrix~$B^*$ with entries 
	given by
	\[
		b^*_{ij} := \begin{cases}
			-1 & \text{if $i\in D$ and $j=i$}; \\
			x_i^{-1} & \text{if $i=i_\ell\in C$ and $j \in \bigl\{ 
			[i_{\ell-1}+1], [i_{\ell-1}+2], \dots, i_\ell \bigr\}$}; \\
			0 & \text{otherwise}.
		\end{cases}
	\]
	For this matrix, if $j$ is an index in~$C$ then $b^*_{\smash{jj}}$ is the 
	only non-zero element in column~$j$, and if $i$ is an index in~$D$, then 
	the only non-zero element on row~$i$ is a~$-1$ on the diagonal. So this 
	time we can bring the matrix in upper triangular form using elementary 
	column operations that leave the diagonal elements alone. We conclude that
	\[
		\det B
		= (-1)^{n-1} \det B^*
		= (-1)^{n-1} \prod_{i=1}^n b^*_{ii}
		= (-1)^{n-1} (-1)^n \prod_{\ell=1}^k \frac1{-x_{i_\ell}}
		< 0.
	\]
	This completes the proof.
\end{proof}

\begin{remark}
	Note that by Lemma~\ref{lem: Motzkin} the system $A\T x=e$ always has a 
	non-negative solution for a matrix~$A$ that satisfies both $Ae>0$ 
	and~\eqref{cnd: non-incr}. So in this case Theorems \ref{thm: singular} 
	and~\ref{thm: invertible} say that $\det A$ is zero if there is more than 
	one closed SCC and strictly positive otherwise. A similar statement can be 
	made if instead of $Ae>0$ we have $Ae<0$: in that case $\det A$ is zero if 
	there is more than one closed SCC and strictly negative otherwise.
\end{remark}

\section{Conditions for a P-matrix}
\label{sec: P-matrix}

Recall that a real $n\times n$ matrix is a $P$-matrix if all its principal 
minors are strictly positive. As we have motivated in Section~\ref{sec: 
application}, we have a special interest in non-negative matrices in the 
class~\eqref{cnd: non-incr} that are $P$-matrices. Of course, we can in 
principle determine whether such a matrix~$A$ is a $P$-matrix by applying 
Lemma~\ref{lem: Motzkin} and Theorems \ref{thm: singular} and~\ref{thm: 
invertible} to each principal submatrix of~$A$. But ideally, we would like to 
find simple conditions that are more easily verified in the setting of our 
application and imply that $A$ is a $P$-matrix. In this section, we provide 
such a condition.

We start by observing that a non-negative matrix~$A$ is a $P$-matrix if all 
inequalities in~\eqref{cnd: non-incr} are strict. Indeed, this is the original 
condition from Motzkin's Theorem (see Section~\ref{sec: main result}) for 
$\det A$ to be strictly positive. That $A$ is actually a $P$-matrix follows 
because every principal submatrix of~$A$ clearly satisfies the same condition. 
In fact, our main results show that it is sufficient to require only that the 
first inequality in~\eqref{cnd: non-incr} is strict. That is, we claim that a 
non-negative $n\times n$ matrix~$A$ is a $P$-matrix if it satisfies
\begin{equation}
	\condition{cnd: strong Motzkin}
	a_{ii} > a_{i[i+1]} \geq a_{i[i+2]} \geq \dots \geq a_{i[i+n-1]},
	\qquad i=1,\dots,n.
\end{equation}
To see this, note that such a matrix has strictly positive row sums and the 
graph~$\Gc(A)$ has only one closed SCC, because there is an edge from every 
vertex~$i$ to the next vertex~$[i+1]$. So $\det A>0$ by Lemma~\ref{lem: 
Motzkin} and Theorem~\ref{thm: invertible}, and $A$ is a $P$-matrix because 
every principal submatrix also satisfies~\eqref{cnd: strong Motzkin}.

But we can do better. Our next theorem provides a condition for a matrix~$A$ 
to be a $P$-matrix that covers the special cases discussed above.

\begin{theorem}
	\label{thm: P-matrix}
	A non-negative real $n\times n$ matrix~$A$ is a $P$-matrix if, in addition 
	to $Ae>0$ and~\eqref{cnd: non-incr}, it also satisfies the following 
	condition:
	\begin{equation}
		\condition{cnd: pre-k gaps}
		\parbox[t]{.9\textwidth}{There is an index~$k$ such that $a_{rr} > 
		a_{rk}$ for all rows~$r$ other than~$k$.}
	\end{equation}
\end{theorem}

\begin{proof}
	The proof consists of two parts: in the first part we show that 
	\eqref{cnd: non-incr} and~\eqref{cnd: pre-k gaps} together imply that 
	$\Gc(A)$ has exactly one closed SCC, so that we have $\det A>0$ by 
	Lemma~\ref{lem: Motzkin} and Theorem~\ref{thm: invertible}; in the second 
	part we show that for a non-negative matrix the properties $Ae>0$, 
	\eqref{cnd: non-incr} and~\eqref{cnd: pre-k gaps} are preserved when we 
	remove one component (i.e., a row and the corresponding column) of the 
	matrix. The two parts combined establish that all principal minors of~$A$ 
	are strictly positive.

	For the first part, suppose $A$ satisfies conditions \eqref{cnd: non-incr} 
	and~\eqref{cnd: pre-k gaps}, and consider the graph~$\Gc(A)$. Let $r$ be a 
	row index other than the index~$k$ from condition~\eqref{cnd: pre-k gaps}. 
	We claim that there is a path in~$\Gc(A)$ from $r$ to~$k$. Indeed, we can 
	construct such a path step by step, as follows. We first set $r_0 := r$. 
	Now assume that after $m$~steps we have obtained a path~$(r_0, \dots, 
	r_m)$ that does not yet visit vertex~$k$. Then by~\eqref{cnd: pre-k gaps} 
	there must be some index~$i$ with $d(i,k) < d(r_m,k)$ such that there is a 
	gap right before the element~$a_{r_mi}$ on row~$r_m$ of~$A$. We set 
	$r_{m+1} := i$, and since there is an edge in~$\Gc(A)$ from $r_m$ to~$i$, 
	we can append this vertex to our path. Continuing like this, the path must 
	eventually visit~$k$ because each added vertex is strictly closer to~$k$ 
	than the previous vertex. So there is a path to~$k$ from every other 
	vertex of~$\Gc(A)$, hence every vertex is either in an open SCC or in the 
	SCC that contains~$k$. It follows that the SCC containing~$k$ is the only 
	closed SCC of~$\Gc(A)$.

	For the second part of the proof, suppose the matrix~$A$ is non-negative 
	and satisfies $Ae>0$, \eqref{cnd: non-incr} and~\eqref{cnd: pre-k gaps}. 
	We claim that these properties are preserved when we remove one component 
	of the matrix. Without loss of generality we can assume we remove 
	component~$n$, provided we allow the index~$k$ from condition~\eqref{cnd: 
	pre-k gaps} to be arbitrary. Let $A'$ be the obtained submatrix. It is 
	easy to see that $A'$ is non-negative, has strictly positive row sums, and 
	satisfies~\eqref{cnd: non-incr}; our only concern is whether 
	condition~\eqref{cnd: pre-k gaps} is preserved. In the case $k<n$ this is 
	clearly true because for $r = 1,\dots,n-1$, both $a_{rr}$ and~$a_{rk}$ are 
	elements of the submatrix~$A'$. It remains to consider the case in which 
	$A$ satisfies condition~\eqref{cnd: pre-k gaps} (only) for $k=n$. But in 
	this case the submatrix~$A'$ satisfies~\eqref{cnd: pre-k gaps} with $k$ 
	replaced by~$1$ because for $r = 2,\dots,n-1$, we have $a_{rr} > a_{rk} = 
	a_{rn} \geq a_{r1}$, using the fact that $A$ satisfies~\eqref{cnd: 
	non-incr}.
\end{proof}

\begin{remark}
	Inspection of the proof reveals that the assumption that $A$ is 
	non-negative in Theorem~\ref{thm: P-matrix} can be weakened to the 
	assumption that for every row of~$A$, the sum of the diagonal entry and 
	all negative entries on that row is strictly positive. Furthermore, the 
	assumptions $A\geq0$ and~$Ae>0$ are only used in the proof to determine 
	the sign of every principal minor; if we replace them by the single 
	assumption that $A$ is a strictly negative matrix, then the proof of the 
	theorem still goes through but shows that all principal minors of~$A$ are 
	strictly negative, i.e., that the matrix $-A$ is a $P$-matrix.
\end{remark}

We close off this section with a brief discussion of conditions that are both 
necessary and sufficient for a non-negative matrix~$A$ in the 
class~\eqref{cnd: non-incr} with positive row sums to be a $P$-matrix. We 
first note that if $A$ has a constant row, then the condition~\eqref{cnd: 
pre-k gaps} is in fact necessary (as well as sufficient) for~$A$ to be a 
$P$-matrix. Indeed, suppose $A$ is a $P$-matrix, let $k$ be the constant row, 
and let $r$ be any row other than~$k$. Then it must be the case that $a_{rr} > 
a_{rk}$, because the alternative is that the $2\times2$ submatrix formed by 
the entries $a_{rr}, a_{rk}$ and $a_{kr}, a_{kk}$ has two constant rows, hence 
determinant zero.

Condition~\eqref{cnd: pre-k gaps} is however not necessary in general for~$A$ 
to be a $P$-matrix. For example, suppose the dimension~$n$ is at least~3 and 
the matrix~$A$ satisfies
\[
	a_{ii} = a_{i[i+1]} > a_{i[i+2]} > \dots > a_{i[i+n-1]} \geq 0
	\qquad i=1,\dots,n.
\]
Then $A$ does not satisfy~\eqref{cnd: pre-k gaps} because $a_{rr} = a_{rk}$ if 
$r = [k-1]$. But $A$ is a $P$-matrix: clearly, every $2\times2$ principal 
minor is strictly positive, and it is easy to see that for every principal 
submatrix~$A'$ of dimension at least $3\times 3$, the row sums are strictly 
positive and the corresponding graph~$\Gc(A')$ has only one SCC, so that $\det 
A'>0$ by Lemma~\ref{lem: Motzkin} and Theorem~\ref{thm: invertible}.

Nonetheless, our observation about condition~\eqref{cnd: pre-k gaps} does 
allow us to formulate a condition that is necessary for a generic non-negative 
$n\times n$ matrix~$A$ in the class~\eqref{cnd: non-incr} with positive row 
sums to be a $P$-matrix: For any row~$r$ of the matrix, let $k_r$ denote the 
largest number in~$\{0,1, \dots, n-1\}$ such that $a_{rr} = a_{r[r+k_r]}$, and 
let $A_r$ be the principal submatrix consisting of the rows and columns of~$A$ 
with indices $r, [r+1], \dots, [r+k_r]$. Then $A_r$ has a constant row. So for 
$A$ to be a $P$-matrix, it is necessary that condition~\eqref{cnd: pre-k gaps} 
holds for each of the principal submatrices~$A_r$, that is, $A$ has to 
satisfy
\begin{equation}
	\condition{cnd: necessary}
	\parbox[t]{.9\textwidth}{For every row~$r$ such that~$k_r>0$, $a_{ss} > 
	a_{sr}$ for $s = [r+1], \dots, [r+k_r]$.}
\end{equation}

The following matrix satisfies~\eqref{cnd: necessary} but is not a 
$P$-matrix:
\[
	\begin{bmatrix}       \,
		2 & 2 & 1 & 1 & 0 \\
		0 & 1 & 1 & 0 & 0 \\
		0 & 0 & 1 & 1 & 1 \\
		1 & 0 & 0 & 1 & 1 \\
		1 & 0 & 0 & 0 & 1 \,
	\end{bmatrix}
\]
So although~\eqref{cnd: necessary} is necessary for a generic non-negative 
matrix~$A$ satisfying $Ae>0$ and~\eqref{cnd: non-incr} to be a $P$-matrix, it 
is not sufficient. We leave finding a simple condition that is both necessary 
and sufficient as an open problem.

\section{Conclusions}
\label{sec: conclusions}

To conclude, we summarize our findings. For a real $n\times n$ matrix~$A$ 
satisfying condition~\eqref{cnd: non-incr}, we explicitly described the set of 
solutions to the system~$A\T x = \lambda\,e$ in terms of fundamental solutions 
and null vectors. This led to our main result, a characterization of the sign 
of the determinant of~$A$ in terms of the number of closed SCCs of the 
graph~$\Gc(A)$ associated with~$A$ and the solutions to the system~$A\T x = 
e$. We also showed that if the row sums of~$A$ are all strictly positive or 
all strictly negative, then the number of closed SCCs determines~$\sgn(\det 
A)$. Finally, we provided an easy-to-verify sufficient condition under which a 
non-negative matrix~$A$ in the class~\eqref{cnd: non-incr} is a $P$-matrix. 
The problem of finding a simple condition that is both necessary and 
sufficient for~$A$ to be a $P$-matrix remains open.

\section*{Appendix}

For the application discussed in Section~\ref{sec: application}, we are 
especially interested in the case in which the matrix~$A$ 
satisfies~\eqref{cnd: non-incr} and $Ae > 0$, and the graph~$\Gc(A)$ has a 
unique closed SCC. This case is covered by Theorem~\ref{thm: invertible} 
together with Lemma~\ref{lem: Motzkin}. However, for this special case we also 
have a more direct proof of the fact that $\det A>0$, which is essentially a 
generalization of Motzkin's original argument~\cite{MR0223384}* {proof of 
Theorem~8}. We present this alternative proof below, as it may be of 
independent interest.

We need extra notation. Given a matrix~$A$ satisfying~\eqref{cnd: non-incr}, 
we first define
\[
	\Gpre_r
	:= \bigl\{\, k\in \{1,\dots,n\}: a_{rk} > a_{r[k+1]} \,\bigr\},
	\qquad r = 1,\ldots,n.
\]
This is the set of all vertices~$k$ in the graph~$\Gc(A)$ such that there is a 
directed edge from vertex~$r$ to vertex~$[k+1]$. We then define
\[
	\textstyle
	\Gpre_{r,0} := \{\, [r-1] \,\}, \qquad
	\Gpre_{r,t+1} := \bigcup_{s\in \Gpre_{r,t}} \Gpre_{[s+1]}
		\quad \text{for $t\geq0$},
\]
so that $\Gpre_{r,t}$ is the set of all vertices~$k$ such that there is path 
of length~$t$ (allowing $t=0$) in~$\Gc(A)$ from vertex~$r$ to vertex~$[k+1]$. 
Lastly, we set
\[
	\textstyle
	\Gpre_{r,\infty} := \bigcup_{t=0}^\infty \Gpre_{r,t}.
\]

Observe that for every vertex~$r$, the set $\{\, [k+1] \colon k \in 
\Gpre_{r,\infty} \}$ necessarily contains a closed SCC of~$\Gc(A)$. Therefore, 
whenever there are two indices $r$ and~$s$ for which $\Gpre_{r,\infty}$ 
and~$\Gpre_{s,\infty}$ are disjoint, the graph~$\Gc(A)$ must have two disjoint 
closed SCCs. On the other hand, if $\Gc(A)$ has at least two closed SCCs, then 
we can find two indices~$r$ and~$s$ such that $\Gpre_{r,\infty}$ 
and~$\Gpre_{s,\infty}$ are disjoint. In other words, $\Gc(A)$ has exactly one 
closed SCC if and only if the following condition holds:
\begin{equation}
	\condition{cnd: connected}
	\textstyle
	\Gpre_{r,\infty} \cap \Gpre_{s,\infty} \neq \empty
		\qquad \text{for any $r,s \in \{1,\dots,n\}$}.
\end{equation}
We use this condition to prove the following theorem:

\begin{theorem}
	\label{thm: invertible again}
	If the real $n\times n$ matrix~$A$ satisfies $Ae>0$, \eqref{cnd: non-incr} 
	and~\eqref{cnd: connected}, then it has a strictly positive determinant.
\end{theorem}

\begin{proof}
	Let $z$ be a vector such that $Az=0$. Our first goal is to prove that this 
	implies $z=0$, hence $\det A\neq0$. By Lemma~\ref{lem: Motzkin}, $e\T z = 
	z_1+\dots+z_n = 0$. Just as in equation~\eqref{eqn: Motzkin} from the 
	proof of Lemma~\ref{lem: Motzkin} this allows us to write the equality 
	$(Az)_{[r+1]}=0$, for any row~$[r+1]$, as
	\begin{equation}
		\label{eq: invertible}
		\sum_{i=1}^n a_{[r+1][r+i]} z_{[r+i]}
		= \sum_{m=1}^{n-1} \bigl( a_{[r+1][r+m]} - a_{[r+1][r+m+1]} \bigr)
			\sum_{i=1}^m z_{[r+i]}
		= 0.
	\end{equation}

	Now suppose for the moment that $\sum_{i=1}^m z_i \leq0$ for $m = 
	1,\dots,n$. We claim that it is then the case that for every $t\geq0$ and 
	all~$\ell$ in the set~$\Gpre_{1,t}$,
	\begin{equation}
		\label{eq: gap induction}
		\sum_{i=1}^{\ell} z_i = 0
		\qquad\text{and}\qquad
		\sum_{i=1}^m z_{[\ell+i]} \leq0
		\quad \text{for $m=1,\dots,n$}.
	\end{equation}
	The proof is by induction in~$t$. For the base case, observe that $\ell = 
	n$ is the only element of~$\Gpre_{1,0}$, and for this value of~$\ell$ we 
	already know that~\eqref{eq: gap induction} is true. For the induction 
	step, take $t\geq1$ and let $\ell$ be an element of~$\Gpre_{1,t}$. By the 
	definition of~$\Gpre_{1,t}$, this means that there is an~$s$ in the 
	set~$\Gpre_{1,t-1}$ such that $\ell$ is an element of~$\Gpre_{[s+1]}$. By 
	the induction hypothesis, we then have
	\[
		\sum_{i=1}^s z_i = 0
		\qquad\text{and}\qquad
		\sum_{i=1}^m z_{[s+i]} \leq 0
		\quad \text{for $m=1,\dots,n$}.
	\]
	Furthermore, if we take~$u$ to be the index in~$\{1,\dots,n\}$ such that 
	$[s+u] = \ell$, then by the definition of~$\Gpre_{[s+1]}$ we also have
	\[
		a_{[s+1][s+u]} > a_{[s+1][s+u+1]}.
	\]
	But in view of \eqref{cnd: non-incr} and~\eqref{eq: invertible} (which 
	holds for any~$r$, so in particular for $r=s$), this implies that 
	$\sum_{i=1}^u z_{[s+i]} = 0$. Using again that $\sum_{i=1}^n z_i = 0$, it 
	follows from this that both in the case $s+u \leq n$ and in the case $s+u 
	> n$,
	\[
		\sum_{i=1}^{\ell} z_i
		= \sum_{i=1}^s z_i + \sum_{i=1}^u z_{[s+i]}
		= 0.
	\]
	This in turn then implies that
	\[
		\sum_{i=1}^m z_{[\ell+i]}
		= \sum_{i=1}^\ell z_i + \sum_{i=1}^m z_{[\ell+i]}
		= \sum_{i=1}^{\ell+m} z_{[i]}
		\leq 0
		\quad \text{for $m = 1,\dots,n$}.
	\]
	This completes the induction step, and establishes~\eqref{eq: gap 
	induction} for all~$\ell$ in~$\Gpre_{1,\infty}$.

	We have just proven that if $z_1+\dots+z_m \leq 0$ for $m = 1,\dots,n$, 
	then $z_1+\dots+z_\ell = 0$ for all~$\ell$ in the set~$\Gpre_{1,\infty}$. 
	But clearly, there is nothing special about the row index~1; using the 
	same argument we can draw an analogous conclusion for  any row of the 
	matrix. That is, for every~$r \in \{1,\dots,n\}$ we can conclude that if 
	we have $\sum_{i=1}^m z_{[r+i]} \leq0$ for $m=1,\dots,n$, then
	\[
		\sum_{i=1}^\ell z_{[r+i]} = 0 \quad
		\text{for all~$\ell \in \{1,\dots,n\}$ such that $[r+\ell] \in 
		\Gpre_{[r+1],\infty}$}.
	\]
	Moreover, by applying the argument to the vector~$-z$ instead of~$z$, we 
	see that we can draw the same conclusion if $\sum_{i=1}^m z_{[r+i]} \geq0$ 
	for $m=1,\dots,n$.

	Now, let $r$ and~$s$ be distinct indices that respectively maximize 
	$\sum_{i=1}^r z_i$ and minimize $\sum_{i=1}^s z_i$. As we showed in the 
	proof of Lemma~\ref{lem: Motzkin}, this implies
	\[
		\sum_{i=1}^m z_{[r+i]} \leq 0 \quad\text{and}\quad
		\sum_{i=1}^m z_{[s+i]} \geq 0 \quad\text{for $m=1,\dots,n$}.
	\]
	Therefore, since $\Gpre_{\smash{ [r+1],\infty }}$ and~$\Gpre_{\smash{ 
	[s+1],\infty }}$ have at least one element~$k$ in common, we can conclude 
	that there exist $u$ and~$v$ in $\{1,\dots,n\}$ such that
	\[
		\sum_{i=1}^u z_{[r+i]} = \sum_{i=1}^v z_{[s+i]} = 0
	\]
	and $[r+u] = [s+v] = k$. But since $\sum_{i=1}^n z_i = 0$, we also have
	\[
		\sum_{i=1}^k z_i
		= \sum_{i=1}^{n+k} z_{[i]}
		= \sum_{i=1}^r z_i + \sum_{i=1}^u z_{[r+i]}
		= \sum_{i=1}^s z_i + \sum_{i=1}^v z_{[s+i]}.
	\]
	It follows that $\sum_{i=1}^r z_i = \sum_{i=1}^s z_i$. By the definition 
	of $r$ and~$s$ and the fact that $\sum_{i=1}^n z_i = 0$, we conclude that 
	$z=0$, hence $\det A \neq 0$.

	It remains to show that the determinant of~$A$ cannot be negative. To this 
	end, we consider the family of matrices $A_t := (1-t)A + tI$, where $t\in 
	[0,1]$ and~$I$ is the $n\times n$ identity matrix. Observe that for every 
	strictly positive value of~$t$, the diagonal element is strictly the 
	largest element on every row of the matrix~$A_t$. It follows that in the 
	directed graph~$\Gc(A_t)$ every vertex~$r$ has a directed edge to 
	vertex~$[r+1]$, so that $\Gpre_{r,\infty} = \{1,\dots,n\}$. So by the 
	result we have just proven, all of the matrices~$A_t$ have non-zero 
	determinant. Since $\det A_t$ is a continuous function of~$t$ and $\det 
	I=1$, it follows that $\det A_t > 0$ for all~$t$ in~$[0,1]$. In 
	particular, $\det A > 0$.
\end{proof}

\end{document}